\documentclass[12pt,a4paper]{article}
\usepackage{amsmath, amsfonts, amsthm,color,latexsym,amssymb}
\usepackage{amscd}
\usepackage{amsmath}
\usepackage{amsfonts}
\usepackage[english]{babel}
\usepackage{amssymb, amsmath, amsfonts, mathrsfs,amsthm, color}
\usepackage{amsmath}
\usepackage{amsfonts}
\usepackage{amssymb}
\usepackage{graphicx}
\usepackage{amssymb}
\usepackage{graphicx}
\providecommand{\U}[1]{\protect\rule{.1in}{.1in}}
\providecommand{\U}[1]{\protect\rule{.1in}{.1in}}
\newtheorem{theorem}{Theorem}[section]

\newtheorem{remark}[theorem]{Remark}

\newtheorem{definition}[theorem]{Definition}
\newtheorem{theoremnew}{Theorem}[subsection]
\newtheorem{propositionnew}{Proposition}[subsection]
\begin{document}

\title{When is the Haar measure a Pietsch measure for nonlinear mappings?}

\author{G. Botelho\thanks{Supported by CNPq Grant 306981/2008-4 and Fapemig Grant
PPM-00295-11.}\,\,, D. Pellegrino\thanks{Supported by CNPq Grant
301237/2009-3}\,\,, P. Rueda\thanks{Supported by Ministerio de Ciencia e
Innovaci\'{o}n MTM2011-22417.}\,\,, J. Santos and\\J. B. Seoane-Sep\'{u}lveda\thanks{Supported by by the Spanish Ministry of
Science and Innovation, grant MTM2009-07848.\hfill\newline2010 Mathematics
Subject Classification: 28C10, 47B10.}}
\date{}

\maketitle

\begin{abstract}
We show that, as in the linear case, the normalized Haar measure on a compact topological
group $G$ is a Pietsch measure for nonlinear summing mappings on closed translation invariant subspaces of $C(G)$. This answers a question posed to the authors by J. Diestel. We also show that our result applies to several well-studied  classes of nonlinear summing mappings. In the final section some problems are proposed.
\end{abstract}


\section{ Introduction}

The Haar measure on a compact topological group $G$ is simply a Radon measure $\sigma_G$
on the Borel sets of $G$ which is translation invariant, that is, $\sigma_G(gB)=\sigma_G(B)$ for every Borel set $B$ and every $g\in G$. A well-known fact,
essentially proved by A. Haar \cite{ha} in 1933 (see also \cite{ca, Neu, Ru,
we}), is that there is only one normalized Haar measure on $G$. In the same
year (in fact in the same issue of Annals of Mathematics), J. von Neumann
\cite{vn} used Haar's Theorem to solve Hilbert's fifth problem in the case of compact groups. The uniqueness of the normalized Haar measure  be used later
in this paper.

A cornerstone in the theory of absolutely summing operators, the Pietsch
Domination Theorem asserts that a continuous linear operator $v\colon
X_{1}\longrightarrow X_{2}$ between Banach spaces is absolutely $p$-summing if
and only if there is a constant $C>0$ and a Borel probability measure $\mu$ on
the closed unit ball of the dual of $X_{1},$ $\left(  B_{X_{1}^{\ast}}%
,\sigma(X_{1}^{\ast},X_{1})\right)  ,$ such that%
\begin{equation}
\left\Vert v(x)\right\Vert \leq C \cdot\left(  \int_{B_{X_{1}^{\ast}}}\left\vert
\varphi(x)\right\vert ^{p}d\mu(\varphi)\right)  ^{\frac{1}{p}} \label{gupdt}%
\end{equation}
for every $x\in X_{1}$. Such a measure $\mu$ is said to be a \textit{Pietsch
measure for }$v$.

When $X_{1}=C(K)$ for some compact Hausdorff space, a simple reformulation of
the Pietsch Domination Theorem tells us that a continuous linear operator
$v\colon C(K)\longrightarrow X_{2}$ is absolutely $p$-summing if and only if
there is a constant $C>0$ and a Borel probability measure $\mu$ on the Borel
sets of $K$ such that%
\[
\left\Vert v(f)\right\Vert \leq C\cdot\left(  \int_{K}\left\vert f(z)\right\vert
^{p}d\mu(z)\right)  ^{\frac{1}{p}}%
\]
for every $f\in C(K)$. The measure $\mu$ above is also called a
\textit{Pietsch measure for }$v$.

The Pietsch Domination Theorem has analogs in different contexts, including
versions for classes of absolutely summing nonlinear operators (see, for
example, \cite{achour, cha, CDo, Dimant, FaJo, SP, joed}). Recently, in
\cite{BoPeRu, PeSa, adv} the concept of \textit{abstract $R$-$S$-abstract
$p$-summing mapping} was introduced in such a way that several previous known
versions of the Pietsch Domination Theorem can be regarded as particular
instances of one single result.

It is interesting to mention that in most of the cases we have almost no structural
information on the Pietsch measures and, as mentioned in \cite[p.\,56]{DiJaTo}, ``\textit{in general its existence is accessible only by
transfinite means}''. Nevertheless in the important case when
$X_{1}=C(G)$ and $G$ is a compact Hausdorff topological group the precise nature of the Pietsch
measure is known: motivated by results from \cite{Gordon, Pelc}, in
\cite[p.\,56]{DiJaTo} it is proved that if $G$ is a compact topological
group, then the normalized Haar measure on $G$ is a Pietsch measure for any
translation invariant $p$-summing linear operator on a closed translation
invariant subspace of $C(G)$ that separates points of $G$ (cf. Theorem \ref{tu}). The particular case in which $G$ is the circle
group $\left\{  z\in\mathbb{C}:\left\vert z\right\vert =1\right\}  $ is used
in \cite{K}.

J. Diestel posed to the authors the question whether, as in the linear case, in
this more general context of \cite{BoPeRu, PeSa} the normalized Haar measure on a compact topological
group $G$ is still a Pietsch measure for any translation invariant
(not necessarily linear) $R$-$S$-abstract $p$-summing mapping on a closed
invariant subspace of $C(G)$. The precise meaning of the terms aforementioned
used shall be clear in the forthcoming section.

In this paper we solve Diestel's question in the positive. We show in Section
\ref{results} that the answer is affirmative provided two natural and general
conditions are satisfied (cf. Theorem \ref{genHaar}). In Section \ref{applications} we show that several usual classes of $R$-$S$-abstract summing (linear and non-linear)
mappings enjoy these two conditions, confirming that Diestel's question has a
positive answer in several important cases. In particular, we improve the original linear result from \cite{DiJaTo} showing that the assumption that $F$ separates points of $G$ can be dropped (cf. Theorem \ref{tugen}). However there are still some open
questions which we detail in the final section.

\section{Preliminaries and background}

Henceforth $G$ denotes a (non necessarily abelian) compact Hausdorff
topological group. The operation on $G$ shall be denoted as multiplicative. The
symbol $C(G)$ stands for the Banach space of continuous functions $f \colon
G\longrightarrow\mathbb{K}$, where $\mathbb{K} = \mathbb{R}$ or $\mathbb{C}$,
endowed with the usual sup norm.

\begin{definition}\rm (a) A non empty set $F\subset C(G)$ is \textit{closed (left)
translation invariant} if%
\[
\{T^{\phi}:T\in F~and~\phi\in G\}\subset F,
\]
where $T^{\phi}(\varphi):=T(\phi\varphi).$\newline(b) Let $X$ be any set and $F$
be a closed translation invariant subset of $C(G)$. A map $u\colon
F\longrightarrow X$ is \textit{translation invariant} if $u(T)=u(T^{\phi})$
for all $\phi\in G$ and $T\in F$.\\
(c) A set $F \subset C(G)$ is said to {\it separate points} of $G$ if for all $x, y \in G$, $x \neq y$, there exists $T \in F$ such that $T(x) \neq T(y)$.
\end{definition}

The following version of the Pietsch Domination Theorem appears in \cite[p.\,56]{DiJaTo} (by $\pi_p(u)$ we denote the $p$-summing norm of the operator $u$):

\begin{theorem}
\label{tu} If $X$ is a Banach space, $G$ is a compact Hausdorff topological
group, $F$ is a closed translation invariant subspace of $C(G)$ that separates points of $G$, and $u \colon
F\longrightarrow X$ is a translation invariant $p$-summing linear operator,
then the normalized Haar measure $\sigma_G$ on $G$ is a Pietsch measure for $u$ in the sense that
$$\|u(f)\| \leq \pi_p(u) \cdot \left(\int_G |f(x)|^p \,d\sigma_G(x)\right)^\frac1p {\rm ~for~every~} f \in F. $$
\end{theorem}

\begin{remark}\rm The assumption that $F$ separates points of $G$ is missing in \cite{DiJaTo}. To understand where this assumption is needed in the proof provided by  \cite{DiJaTo}, see Subsection \ref{linearcase}. In that same subsection we provide a proof that does not require this extra assumption.
\end{remark}

Next, we describe the general Pietsch Domination Theorem proved in
\cite{BoPeRu,PeSa}. Let $X$, $Y$ and $E$ be (arbitrary) sets, and $\mathcal{H}$ be
a family of mappings from $Y$ to $X.$ Let also $Z$ be a Banach space and $K$
be a compact Hausdorff topological space. Assume that $S\colon{\mathcal{H}%
}\times E\times Z\longrightarrow\lbrack0,\infty)$ is an arbitrary map and
$R\colon K\times E\times Z\longrightarrow\lbrack0,\infty)$ is so that%

\[
R_{x,b}(\varphi):=R(\varphi,x,b)
\]
is continuous on $K$ for all $(x,b)\in E\times Z.$ For $0<p<\infty,$ a mapping
$f\in{\mathcal{H}}$ is said to be \textit{$R$-$S$-abstract $p$-summing} if
there is a constant $C>0$ so that%
\[
\left(  \sum_{j=1}^{m}S(f,x_{j},b_{j})^{p}\right)  ^{\frac{1}{p}}\leq C\cdot\left(
\sup_{\varphi\in K}\sum_{j=1}^{m}R\left(  \varphi,x_{j},b_{j}\right)
^{p}\right)  ^{\frac{1}{p}},
\]
for all $\left(  x_{j},b_{j}\right)  \in E\times Z,$ $j=1,\ldots,m$ and
$m\in\mathbb{N}$. The infimum of such constants $C$ is denoted by $\pi
_{RS,p}(f)$. The Unified Pietsch Domination Theorem \cite{BoPeRu,PeSa} reads as follows:

\begin{theorem}
\label{gen} Let $R$ and $S$ be as above, $0<p<\infty$ and $f\in{\mathcal{H}}$.
Then $f$ is $R$-$S$-abstract $p$-summing if and only if there is a constant
$C>0$ and a regular Borel probability measure $\mu$ on $K$ such that%
\begin{equation}
S(f,x,b)\leq C\cdot\left(  \int_{K}R\left(  \varphi,x,b\right)  ^{p}d\mu\left(
\varphi\right)  \right)  ^{\frac{1}{p}} \label{2}%
\end{equation}
for all $(x,b)\in E\times Z.$ Moreover, the infimum of such constants $C$
equals $\pi_{RS,p}(f)$. Such a measure $\mu$ is called a $R$-$S$-abstract
measure for $f$.
\end{theorem}

Thus, J. Diestel's question concerns the validity of Theorem \ref{tu} in the context
of Theorem \ref{gen}.

\section{Main result}

\label{results}

In this section we shall keep the notation of the previous one. In our main
result (Theorem \ref{genHaar} below) we have identified two natural conditions
on $R$ and $S$ under which Diestel's question has a positive answer:

\begin{theorem}
\label{genHaar}Let $R$ and $S$ be as above, $0<p<\infty$, and $u \colon Y\longrightarrow
X$ be a map in $\mathcal{H}$. Let us further assume that

{\rm (i)} $K=G$ is a compact topological group and $E$ is a closed translation invariant subspace of $C(G),$

{\rm (ii)} $R(\varphi,T^{\phi},b)\leq R(\phi\varphi,T,b)$ for all $\left(
T,b\right)  \in$ $E\times Z$ and $\varphi,\phi\in G,$ and

{\rm (iii)} The map $S_{u,b}\colon E\longrightarrow(0,\infty)~$defined by
$S_{u,b}(T)=S(u,T,b)$ is translation invariant for every $b\in Z.$

If $u$ is a $R$-$S$-abstract $p$-summing mapping, then the normalized Haar
measure on $G$ is a $R$-$S$-abstract measure for $u$.
\end{theorem}

\begin{proof}
By Theorem \ref{gen} there exists a Borel probability measure $\mu$ on $G$
such that
\[
S(u,T,b)\leq\pi_{RS,p}(u)\cdot\left(  \int_{G}R(\varphi,T,b)^{p}\,d\mu
(\varphi)\right)  ^{1/p}
\]
for every $(T,b) \in E \times Z$. For each $\phi\in G$ define $\mu_{\phi}\in C(G)^{\ast}$ by
\[
\langle\,\mu_{\phi},T \,\rangle :=\langle\,\mu,T^{\phi} \,\rangle
\]
for every $T\in C(G)$. Indeed
$\mu_{\phi}$ is a probability measure as
\[
\langle\,\mu_{\phi},1_{G}\,\rangle= \langle\,\mu,1_{G}^{\phi} \,\rangle=\int_{G}1_{G}^{\phi}\,d\mu=\mu(G)=1,
\]
where $1_{G}$ denotes the constant mapping taking the value $1$. Now we
prove that $\mu_{\phi}$ is a $R$-$S$-abstract measure for $u$: given $(T,b) \in E \times Z$,
\begin{align*}
S(u,T,b)^{p}  &  =S_{u,b}(T)^{p} =S_{u,b}(T^{\phi})^{p}  =S(u,T^{\phi},b)^{p}\\
&  \leq\pi_{RS,p}(u)^{p}\cdot\int_{G}R(\varphi,T^{\phi},b)^{p}\,d\mu(\varphi)\\
&  \leq\pi_{RS,p}(u)^{p}\cdot\int_{G}R(\phi\varphi,T,b)^{p}\,d\mu(\varphi)\\
&  =\pi_{RS,p}(u)^{p}\cdot\int_{G}R_{T,b}(\phi\varphi)^{p}\,d\mu(\varphi)\\
&  =\pi_{RS,p}(u)^{p}\cdot\int_{G}R_{T,b}^{\phi}(\varphi)^{p}\,d\mu(\varphi)\\
&  =\pi_{RS,p}(u)^{p}\langle\,\mu,\left(  R_{T,b}^{p}\right)  ^{\phi}\rangle  =\pi_{RS,p}(u)^{p}\langle\,\mu_{\phi},R_{T,b}^{p}\rangle\\
&  =\pi_{RS,p}(u)^{p}\cdot\int_{G}R(\varphi,T,b)^{p}\,d\mu_{\phi}(\varphi).
\end{align*}
Let $\sigma_{G}$ denote the normalized Haar measure on $G$. As the map
$\phi\mapsto\mu_{\phi}$ is continuous when $C(G)^{\ast}$ is endowed with the
weak*-topology, there exists a measure $\nu\in C(G)^{\ast}$ such that
\[
\langle\,\nu,T\, \rangle=\int_{G}\langle\,\mu_{\phi},T \,\rangle\,d\sigma_{G}(\phi)
\]
for every $T \in C(G)$ (see \cite[p.\,57]{DiJaTo}). If $T\in E$, $T\geq0$, then
\[
\langle\,\nu,T \,\rangle=\int_{G}\langle\,\mu_{\phi},T\,\rangle d\sigma_{G}(\phi)=\int_{G}\langle\,\mu,T^{\phi
}\,\rangle\,d\sigma_{G}(\phi)\geq0,
\]
so $\nu$ is a non-negative regular Borel measure on $G$. Moreover, the following calculation shows that $\nu$ is a probability:
\[
\langle\,\nu,1_{G}\rangle=\int_{G}\langle\,\mu_{\phi},1_{G}\rangle\,d\sigma_{G}(\phi)=\int_{G}%
1_{G}\,d\sigma_{G}=\langle\,\sigma_{G},1_{G}\rangle=1.
\]
Since%
\[
\frac{S(u,T,b)^{p}}{\pi_{RS,p}(u)^{p}}\leq\int_{G}R_{T,b}(\varphi)^{p}%
\,d\mu_{\phi}(\varphi)=\langle\,\mu_{\phi},R_{T,b}^{p}\,\rangle,
\]
for every $(T, b) \in E \times Z$, it follows that
\begin{align*}
\frac{S(u,T,b)^{p}}{\pi_{RS,p}(u)^{p}}  &  =\frac{S(u,T,b)^{p}}{\pi
_{RS,p}(u)^{p}}\int_{G}1_{G}\,d\sigma_{G}  =\int_{G}\frac{S(u,T,b)^{p}}{\pi_{RS,p}(u)^{p}}1_{G}\,d\sigma_{G}\\
&  \leq\int_{G}\langle\,\mu_{\phi},R_{T,b}^{p}\,\rangle\,d\sigma_{G}(\phi)  =\langle\,\nu,R_{T,b}^{p}\,\rangle  =\int_{G}R_{T,b}(\phi)^{p}\,d\nu(\phi),
\end{align*}
for every $(T, b) \in E \times Z$. Therefore $\nu$ is a $R$-$S$-abstract measure for $u$. Next we prove that $\nu$ is translation invariant: given $T \in C(G)$ and $\phi_0 \in G$,
\begin{align*}
\langle\,\nu,T^{\phi_{0}}\,\rangle  &  =\int_{G}\langle\,\mu_{\phi},T^{\phi_{0}}\,\rangle\, d\sigma_{G}%
(\phi)  =\int_{G}\int_{G} T^{\phi_{0}}(\varphi)\, d\mu_{\phi}(\varphi)\,
d\sigma_{G}(\phi)\\
&  =\int_{G}\int_{G} T(\phi_{0}\phi\varphi)\, d\mu(\varphi)\, d\sigma_{G}%
(\phi)  =\int_{G}\int_{G} T(\phi_{0}\phi\varphi)\, d\sigma_{G}(\phi)\, d\mu
(\varphi)\\
&  =\int_{G}\int_{G} T(\phi\varphi)\, d\sigma_{G}(\phi)\, d\mu(\varphi)  =\int_{G}\int_{G} T(\phi\varphi)\, d\mu(\varphi)\, d\sigma_{G}(\phi)\\
&  =\int_{G}\int_{G} T(\varphi)\, d\mu_{\phi}(\varphi)\, d\sigma_{G}(\phi)  =\int_{G}\langle\,\mu_{\phi},T\,\rangle\, d\sigma_{G}(\phi)  = \langle\,\nu,T\,\rangle.
\end{align*}
By the uniqueness of the normalized Haar measure $\sigma_{G}$ we conclude that
$\nu=\sigma_{G}$.
\end{proof}

Note that the condition of translation invariance is imposed to the map
$S_{u,b}(\cdot):=S(u,\cdot,b)$ instead of to the map $u.$ Thus, although in the
applications we always have the information that $u$ is translation invariant,
\textit{a priori} our abstract result does not need this hypothesis. An
extremal example shows that in fact this choice seems to be adequate: if $S$
is the null mapping then obviously no hypothesis on $u \colon Y\longrightarrow X$ is needed.

\section{Applications\label{appp}}

\label{applications} In this section we show that Theorem \ref{genHaar}
applies to several usual classes of $R$-$S$- abstract summing mappings, including some well-studied classes of (linear and nonlinear) absolutely summing mappings.

We shall use several times that if $G$ is a compact Hausdorff space and $F$ is a closed subspace of $C(G)$, then
$$\sup_{\alpha \in B_{F^*}} \sum_{j=1}^m |\alpha(T_j)|^p = \sup_{\alpha \in B_{C(G)^*}} \sum_{j=1}^m |\alpha(T_j)|^p = \sup_{x \in G} \sum_{j=1}^m |T_j(x)|^p,$$
for all $T_1, \ldots, T_m \in F$. The first equality follows from the Hahn--Banach Theorem and the second follows from a canonical argument using {\it point masses} (cf. \cite[page 41]{DiJaTo}).

\subsection{Absolutely summing linear operators}\label{linearcase}

Let us see that, for linear operators, Theorem \ref{genHaar} recovers and generalizes Theorem
\ref{tu}. First of all let us see that the assumption that $F$ separates points of $G$ is crucial in the proof that \cite{DiJaTo} provides for Theorem
\ref{tu}.

\begin{propositionnew}
Let $K$ a compact Hausdorff space and let $F$ be a closed subspace of $C(K)$ that separates points of $K$. If $Y$ is a Banach space and $u \colon F \longrightarrow Y$ is a $p$-summing linear operator, then there is a Borel probability measure $\mu$ on $K$ such that
$$\|u(f)\|^p \leq \pi_p(u)^p \cdot \int_K |f(x)|^p d\mu(x){\rm ~for~every~} f \in F. $$
\end{propositionnew}

\begin{proof} Consider the mapping
 $$\delta \colon K \longrightarrow F^*~,~\delta(x) = \delta_x \colon F \longrightarrow \mathbb{K}~,~\delta_x(f) = f(x).$$
It is easy to see that (i) $\delta(K) \subseteq B_{F^*}$,
 (ii) $\delta(K)$ is norming for $F$, (iii) $\delta(K)$ is weak* compact in $F^*$. So, by the Pietsch Domination Theorem \cite[Theorem 2.12]{DiJaTo} there is a Borel probability measure $\nu$ on $(\delta(K), w^*)$ such that
$$\|u(f)\|^p \leq \pi_p(u)^p \cdot \int_{\delta(K)} |y(f)|^p d\nu(y){\rm ~for~every~} f \in F. $$
Since $F$ separates points of $K$, it follows that $\delta$ is injective and the inverse function $\delta^{-1}\colon \delta(K) \longrightarrow K $
is obviously measurable considering the Borel sets in $K$ and in $(\delta(K), w^*)$. Let $\mu$ be the image measure with respect to $\delta^{-1}$, that is, $\mu$ is a Borel measure on $K$ and $\mu(A) = \nu(\delta(A))$. It is clear that $\mu$ is a probability measure and by \cite[Proposition 9.1]{Folland} we have that
\begin{align*}\int_K |f(x)|^p d\mu(x) &= \int_{\delta(K)} |f(\delta^{-1}(y))|^p d\nu(y) = \int_{\delta(K)} |\delta(\delta^{-1}(y))(f)|^p d\nu(y)\\
&= \int_{\delta(K)} |y(f)|^p d\nu(y) \geq \frac{\|u(f)\|^p}{\pi_p(u)^p}
\end{align*}
for every $f \in F$.
\end{proof}

The proof of the proposition above makes clear that the assumption that $F$ separates points of $G$ is needed to validate the assertion made in the first three lines of the proof of \cite[Theorem, page 56]{DiJaTo}. Let us see that this assumption can be dropped with the help of Theorem \ref{genHaar}:

\begin{theoremnew}
\label{tugen} If $X$ is a Banach space, $G$ is a compact Hausdorff topological
group, $F$ is a closed translation invariant subspace of $C(G)$ and $u \colon
F\longrightarrow X$ is a translation invariant $p$-summing linear operator,
then the normalized Haar measure $\sigma_G$ on $G$ is a Pietsch measure for $u$ in the sense that
$$\|u(f)\| \leq \pi_p(u) \cdot \left(\int_G |f(x)|^p \,d\sigma_G(x)\right)^\frac1p {\rm ~for~every~} f \in F. $$
\end{theoremnew}

\begin{proof} Make the following choice for the parameters of Theorem  \ref{genHaar}:
\[
E=F=Y,~K=G,~Z=\mathbb{K}, ~\mathcal{H}={\mathcal{L}}(F;X),
\]
\[
R \colon G\times F\times\mathbb{K}\longrightarrow\lbrack0,\infty)~,~
R(\varphi,T,b)=|T(\varphi)|,\mathrm{~and}%
\]
\[
S\colon{\mathcal{L}}(F;X)\times F\times\mathbb{K}\longrightarrow
\lbrack0,\infty)~,~ S(v,T,b)=\Vert v(T)\Vert.
\]
Given $T_1, \ldots, T_m \in F$ and $b_1, \ldots, b_m \in \mathbb{K}$,
\begin{align*}\sum_{j=1}^m S(u,T_j,b_j)^p &= \sum_{j=1}^m \|u(T_j)\|^p \leq \pi_p(u)^p \cdot \sup_{\alpha \in B_{F^*}} \sum_{j=1}^m |\alpha(T_j)|^p \\& = \pi_p(u)^p \cdot \sup_{x \in G}\sum_{j=1}^m |T_j(x)|^p=  \pi_p(u)^p \cdot \sup_{x \in G}\sum_{j=1}^m R(x,T_j,b_j)^p,
\end{align*}
which proves that $u$ is $R$-$S$-abstract $p$-summing. Let us see that the
conditions of Theorem \ref{genHaar} are satisfied:
\[
R(\varphi,T^{\phi},b)=|T^{\phi}(\varphi)|=|T(\phi\varphi)|=R(\phi
\varphi,T,b),
\]
for all $T\in F$, $\varphi,\phi\in G$ and $b\in Z$; and using that $u$ is
translation invariant,
\[
S_{u,b}(T)=S(u,T,b)=\Vert u(T)\Vert=\Vert u(T^{\phi})\Vert=S(u,T^{\phi
},b)=S_{u,b}(T^{\phi}),
\]
for all $\phi\in G$ and $b\in Z$. So the normalized Haar measure $\sigma_{G}$
on $G$ is a $R$-$S$-abstract measure for $u$. Then
\begin{align*}
\Vert u(T)\Vert^{p}  &  =S(u,T,b)^{p} \leq\pi_{RS,p}(u)^{p}\cdot\int_{G}%
R(\phi,T,b)^{p}\,d\sigma_{G}(\phi) \\&=\pi_{RS,p}(u)^{p}\cdot\int_{G}|T(\phi
)|^{p}\,d\sigma_{G}(\phi),
\end{align*}
for every $T\in F$; therefore $\sigma_{G}$ is a Pietsch measure for $u$.
\end{proof}

\subsection{Dominated homogeneous polynomials}

For the definition of dominated homogeneous polynomials and the corresponding
Pietsch Domination Theorem, see \cite[Definition 3.2 and Proposition
3.1]{anais} or, without proof, \cite[Definition 2.1 and Theorem 3.3]{jfa}.

Let $F$ be a closed translation invariant subspace of $C(G)$ and let $P\colon
F\longrightarrow X$ be a translation invariant $p$-dominated $n$-homogeneous
polynomial. Choose%
\[
E=F=Y,~K=G,~Z=\mathbb{K},
\]
$\mathcal{H}$ to be the Banach space ${\mathcal{P}}(^{n} F;X)$ of continuous
$n $-homogeneous polynomials from $F$ to $X$ with the usual sup norm,
\[
R \colon G\times F\times\mathbb{K}\longrightarrow\lbrack0,\infty
)~,~R(\varphi,T,b)=|T(\varphi)|, \mathrm{~and}%
\]%
\[
S \colon{\mathcal{P}}(^{n}F;X)\times F\times\mathbb{K}\longrightarrow
\lbrack0,\infty)~,~S(Q,T,b)=\Vert Q(T)\Vert^{1/n}.
\]
Since $P$ is $p$-dominated, there is a constant $C$ such that, for all $T_1, \ldots, T_n \in F$ and $b_1, \ldots, b_m \in \mathbb{K}$,%
\begin{align*} \sum_{i=1}^{k} S(P,T_{i},b_{i})^{p} &= \sum_{i=1}^{k}\Vert P(T_{i})\Vert^{p/n} \leq C^p \cdot \sup_{\varphi\in B_{F^{\ast}}}\sum
_{i=1}^{k}|\varphi(T_{i})|^{p} = C^p \cdot \sup_{\varphi\in K}%
\sum_{i=1}^{k} |T_{i}(\varphi)|^{p}\\
& = C^p \cdot  \sup_{\varphi\in K}\sum
_{i=1}^{k} R(\varphi_{i},T_{i},b_{i})^{p}.
\end{align*}
%
So $P$ is $R$-$S$-abstract $p$-summing. Note also that
\[
R(\varphi,T^{\phi},b)=|T^{\phi}(\varphi)|=|T(\phi\varphi)|=R(\phi\varphi,T,b)
\]
for all $T\in F$, $\varphi,\phi\in G$ and $b\in Z$, and using that $P$ is
translation invariant,
\[
S_{P,b}(T)=S(P,T,b)=\Vert P(T)\Vert^{1/n}=\Vert P(T^{\phi})\Vert
^{1/n}=S(P,T^{\phi},b)=S_{P,b}(T^{\phi}),
\]
for all $T \in F$ and $b \in Z$. By Theorem \ref{genHaar} we obtain that the
normalized Haar measure $\sigma_{G}$ on $G$ is a $R$-$S$-abstract measure for
$P$. Then
\begin{align*}
\Vert P(T)\Vert^{p}  &  =S(P,T,b)^{p} \leq\pi_{RS,p}(P)^{p}\cdot \int_{G}%
R(\phi,T,b)^{p}\,d\sigma_{G}(\phi)\\& =\pi_{RS,p}(P)^{p}\cdot \int_{G}|T(\phi
)|^{p}\,d\sigma_{G}(\phi),
\end{align*}
for every $T\in F$; therefore $\sigma_{G}$ is a Pietsch measure for $P$.

\subsection{$\alpha$-subhomogeneous mappings}

For the definition of $\alpha$-subhomogeneous mappings and the corresponding
Pietsch Domination Theorem we refer to \cite[Definition 3.1 and Theorem
2.4]{mona}.

Let $F$ be a closed translation invariant subspace of $C(G)$ and let $f \colon
F\longrightarrow X$ be a translation invariant $\alpha$-subhomogeneous
mapping. Choose%
\[
E=F=Y,~ K=G,~Z=\mathbb{K},~\mathcal{H}=\left\{  h\colon F\longrightarrow
Y:\text{ }h\text{ is }\alpha\text{-subhomogeneous}\right\}  ,
\]
\[
R\colon G\times F\times\mathbb{K}\longrightarrow\lbrack0,\infty)~,~R(\varphi
,T,b)=|T(\varphi)|,
\]%
\[
S \colon\mathcal{H}\times F\times\mathbb{K}\longrightarrow\lbrack
0,\infty)~,~S(h,T,b)=\Vert h(T)\Vert^{1/\alpha}.
\]
Since $f$ is $\alpha$-subhomogeneous, there is a constant $C$ such that, for all $T_1, \ldots, T_k \in F$ and $b_1, \ldots, b_k \in \mathbb{K}$,
\begin{align*} \sum_{i=1}^{k} S(f,T_{i},b_{i})^{p}& = \sum_{i=1}^{k}\Vert f(T_{i})\Vert^{p/\alpha} \leq C^p \cdot \sup_{\varphi\in B_{F^{\ast}}}\sum
_{i=1}^{k}|\varphi(T_{i})|^{p} = C^p \cdot \sup_{\varphi\in
K}\sum_{i=1}^{k} |T_{i}(\varphi)|^{p}\\
& = C^p \cdot \sup_{\varphi\in K}\sum_{i=1}^{k}
R(\varphi_{i},T_{i},b_{i})^{p}.
\end{align*}
So $f$ is $R$-$S$-abstract $p$-summing. Note also that
\[
R(\varphi,T^{\phi},b)=|T^{\phi}(\varphi)|=|T(\phi\varphi)|=R(\phi
\varphi,T,b),
\]
for every $T\in F$, $\varphi,\phi\in G$ and $b\in Z$, and using that $f$ is
translation invariant,
\[
S_{f,b}(T)=S(f,T,b)=\Vert f(T)\Vert^{1/\alpha}=\Vert f(T^{\phi})\Vert
^{1/\alpha}=S(f,T^{\phi},b)=S_{f,b}(T^{\phi}),
\]
for every $T \in F$ and $b \in Z$. By Theorem \ref{genHaar} we obtain that the
normalized Haar measure $\sigma_{G}$ on $G$ is a $R$-$S$-abstract measure for
$f$. Then
\begin{align*}
\Vert f(T)\Vert^{\frac{p}{\alpha}}  &  =S(f,T,b)^{p} \leq\pi_{RS,p}(f)^{p}\cdot%
\int_{G}R(\phi,T,b)^{p}\,d\sigma_{G}(\phi) \\&=\pi_{RS,p}(f)^{p}\cdot\int_{G}%
|T(\phi)|^{p}\,d\sigma_{G}(\phi),
\end{align*}
for every $T\in F$; therefore $\sigma_{G}$ is a Pietsch measure for $f$.

\subsection{Absolutely summing arbitrary mappings}\label{arbitrary}


Let $X$ and $F$ be Banach spaces. Following \cite%
[Definition 2.1]{BoPeRu} (see also  \cite[Definition 3.1]{matos}), an arbitrary mapping $f\colon F\longrightarrow X$ is {\it
absolutely $p$-summing at $a\in F$} if there is a
$C\geq0$ so that%
\[%
{\displaystyle\sum\limits_{j=1}^{m}}
\left\Vert f(a+x_{j})-f(a)\right\Vert ^{p}\leq C\cdot\sup_{\varphi\in B_{E^{\prime
}}}%
{\displaystyle\sum\limits_{j=1}^{m}}
\left\vert \varphi(x_{j})\right\vert ^{p}%
\]
for every natural number $m$ and all $x_{1},\ldots,x_{m} \in F$. The Pietsch Domination Theorem for absolutely $p$-summing mappings at $a \in F$ can be found in \cite[Theorem 4.2]{BoPeRu}.

Assume that $F$ is a closed translation invariant subspace of $C(G)$ and let $%
f\colon F\longrightarrow X$ be a translation invariant absolutely
$p$-summing mapping at $0\in F$. Choose%
\begin{equation*}
E=F=Y,~K=G,~Z=\mathbb{K},~\mathcal{H}=X^{F}=\left\{ h\colon F\longrightarrow
X\right\} ,
\end{equation*}%
\begin{equation*}
R\colon G\times F\times \mathbb{K}\longrightarrow \lbrack 0,\infty
),~~R(\varphi ,T,b)=|T(\varphi )|,
\end{equation*}%
\begin{equation*}
S\colon \mathcal{H}\times F\times \mathbb{K}\longrightarrow \lbrack 0,\infty
)~,~S(h,T,b)=\Vert h(T)-h(0)\Vert .
\end{equation*}%
%
%
%
%
Since $f$ is absolutely $p$-summing at $0 \in F$, there is a constant $C$ such that, for all $T_1, \ldots, T_k \in F$ and $b_1, \ldots, b_k \in \mathbb{K}$,
\begin{align*} \sum_{i=1}^{k}S(f,T_{i},b_{i})^{p}&= \sum_{i=1}^{k}\Vert f(T_{i})-f(0)\Vert ^{p} \leq C^p \cdot \sup_{\varphi \in B_{F^{\ast }}}\sum_{i=1}^{k}|\varphi (T_{i})|^{p}\\
&= C^p \cdot \sup_{\varphi \in
K}\sum_{i=1}^{k}|T_{i}(\varphi )|^{p} = C^p \cdot \sup_{\varphi \in
K}\sum_{i=1}^{k}R(\varphi _{i},T_{i},b_{i})^{p}.
\end{align*}
%
So $f$ is $R$-$S$-abstract $p$-summing. Note also that
\begin{equation*}
R(\varphi ,T^{\phi },b)=|T^{\phi }(\varphi )|=|T(\phi \varphi )|=R(\phi
\varphi ,T,b)
\end{equation*}%
for all $T\in F$, $\varphi ,\phi \in G$ and $b\in Z$. Using that $f$ is
translation invariant,%
\begin{align*}
S_{f,b}(T^{\phi })& =S(f,T^{\phi },b)=\Vert f(T^{\phi })-f(0)\Vert \\
& =\Vert f(T)-f(0)\Vert =S(f,T,b)=S_{f,b}(T),
\end{align*}%
for all $T\in F$ and $b\in Z$. By Theorem \ref{genHaar} we obtain that the
normalized Haar measure $\sigma _{G}$ on $G$ is a $R$-$S$-abstract measure
for $f$. Then
\begin{align*}
\Vert f(T)-f(0)\Vert ^{p}& =S(f,T,b)^{p}\leq \pi
_{RS,p}(f)^{p}\cdot\int_{G}R(\phi ,T,b)^{p}\,d\sigma _{G}(\phi ) \\
& =\pi _{RS,p}(f)^{p}\cdot\int_{G}|T(\phi )|^{p}\,d\sigma _{G}(\phi ),
\end{align*}%
for every $T\in F$; therefore $\sigma _{G}$ is a Pietsch measure for $f$.

\section{Open Problems}

For the definition of {\it Lipschitz $p$-summing mappings} and the corresponding Pietsch Domination Theorem we refer to \cite{FaJo, BoPeRu}.

\medskip

\noindent {\bf Problem.} Let $F$ be a closed translation invariant subspace of $C(G)$, let $X$ be a metric space and $f \colon F\longrightarrow X$ be a translation invariant Lipschitz $p$-summing mapping. Is the Haar measure $\sigma_G$ a Pietsch measure for $f$?

\bigskip

In the case of absolutely summing arbitrary mappings (Subsection \ref{arbitrary}) we assumed that the translation invariant mapping $f \colon F \subseteq C(G) \longrightarrow X$ is absolutely $p$-summing at the origin. What about translation invariant mappings that are absolutely $p$-summing at some $0 \neq T_0 \in F$? We say that a vector $T_0 \in F$ is {\it translation invariant} if $T_{0}^{\phi }=T_{0}$ for every $\phi \in G.$ Let $T_0\in F$. Define
$$f_{T_0} \colon F \longrightarrow X~,~f_{T_0}(T) = f(T + T_0).  $$
It is easy to see that $f$ is absolutely $p$-summing at $T_0$  if and only if $f_{T_0}$ is absolutely $p$-summing at the origin. Besides, if $T_0$ is translation invariant then $f$ is translation invariant  if and only if $f_{T_0}$ is translation invariant. Thus, if we assume that $f$ is translation invariant and absolutely $p$-summing at $T_0$, then $f_{T_0}$ is translation invariant and absolutely $p$-summing at the origin. Therefore the Haar measure is a Pietsch measure for $f_{T_0}$ and hence for $f$.
Anyway we need the extra assumption that the vector $T_0$ is translation invariant. Can this assumption be dropped?

\medskip

\noindent {\bf Problem.} Let $F$ be a closed translation invariant subspace of $C(G)$ and let $%
f\colon F\longrightarrow X$ be a translation invariant mapping that is  absolutely
$p$-summing at some vector of $F$. Is the Haar measure $\sigma_G$ a Pietsch measure for $f$?

\bigskip
\bigskip

\noindent\textbf{Acknowledgements.} The authors thank J. Diestel for drawing our attention to this subject.

\bigskip

\noindent[Geraldo Botelho] Faculdade de Matem\'{a}tica, Universidade Federal
de Uberl\^{a}ndia, 38.400-902 -- Uberl\^{a}ndia -- Brazil, e-mail: botelho@ufu.br

\medskip

\noindent[Daniel Pellegrino] Departamento de Matem\'{a}tica, Universidade
Federal da Para\'{\i}ba, 58.051-900 -- Jo\~{a}o Pessoa -- Brazil, e-mail: dmpellegrino@gmail.com

\medskip

\noindent[Pilar Rueda] Departamento de An\'{a}lisis Matem\'{a}tico, Universidad de
Valencia, 46.100 -- Burjasot -- Valencia -- Spain, e-mail: pilar.rueda@uv.es

\medskip

\noindent[Joedson Santos] Departamento de Matem\'{a}tica, Universidade Federal
de Sergipe, 49.500-000 -- Itabaiana -- Brazil, e-mail: joedsonsr@yahoo.com.br.

\medskip

\noindent[Juan Benigno Seoane-Sep\'{u}lveda] Departamento de An\'{a}lisis
Matem\'{a}tico, Facultad de Ciencias Matem\'{a}ticas, Plaza de Ciencias 3,
Universidad Complutense de Madrid, Madrid -- 28040 -- Spain, e-mail:jseoane@mat.ucm.es

\end{document}